\documentclass[11pt]{article}

\usepackage{a4wide}
\usepackage{amssymb}
\usepackage{amsmath}
\usepackage{latexsym}
\usepackage{epsfig}
\usepackage{mdwlist}
\usepackage{verbatim}  
\usepackage{color}
\usepackage{scrtime}
\usepackage[a4paper,ps2pdf]{hyperref}               

\newcommand{\Grass}[2]{\mathrm{Grass}(#1,#2)}

\def \Grasspn {\mathrm{Grass}(p,n)}

\newcommand{\ST}[2]{\mathbb{C}_*^{#2\times #1}}
\def \STpn {\mathbb{C}_*^{n\times p}}

\newcommand{\eqcllo}[1]{\lfloor #1 \rfloor}
\newcommand{\ecl}[1]{\lfloor #1 \rfloor}
\newcommand{\eclbig}[1]{\left\lfloor #1 \right\rfloor}

\def \rr {\mathbb{R}}
\def \cc {\mathbb{C}}
\def \bbRP {\mathbb{RP}}
\def \bbP {\mathbb{P}}

\def \calD {{\cal{D}}}

\def \calN {{\cal{N}}}

\def \calS {{\cal{S}}}

\def \calY {{\cal{Y}}}
\def \calX {{\cal{X}}}
\def \calV {{\cal{V}}}

\def \rank {\mathrm{rank}}

\def \dist {\mathrm{dist}}
\def \T {\boldsymbol{T}}

\def \diag {\mathrm{diag}}

\def \sep {\mathrm{sep}}

\newcommand{\pacomm}[1]{}
\newcommand{\tech}[1]{}

\newenvironment{proof}{{\it \noindent Proof.}}{\hspace{\stretch{1}} $\square$}

\newtheorem{thrm}{Theorem}[section]
\newtheorem{lmm}[thrm]{Lemma}

\newtheorem{prpstn}[thrm]{Proposition}

\newtheorem{lgrthm}[thrm]{Algorithm}
\newtheorem{dfntn}[thrm]{Definition}

\title{Two-sided Grassmann-Rayleigh quotient iteration\footnotemark[1]}
\author{P.-A. Absil\footnotemark[2]\and P.
Van Dooren\footnotemark[2]}
\date{\tt Submitted 
for publication on 19 Apr 2007}
\begin{document}

\maketitle



\renewcommand{\thefootnote}{\fnsymbol{footnote}} \footnotetext[1]{This
  paper presents research results of the Belgian Network DYSCO
  (Dynamical Systems, Control, and Optimization), funded by the
  Interuniversity Attraction Poles Programme, initiated by the
  Belgian State, Science Policy Office. The scientific responsibility
  rests with its authors. 
This work was supported by the US National
Science Foundation under Grant OCI-0324944
and by the School of Computational Science of Florida State University
through a postdoctoral fellowship.}
\footnotetext[2]{Department of Mathematical Engineering, Universit\'e
  catholique de Louvain, Avenue Georges
  Lema\^{\i}tre 4, B-1348 Louvain-la-Neuve, Belgium ({\tt
    http://www.inma.ucl.ac.be/\{$\sim$absil,$\sim$vdooren\}}).}
\renewcommand{\thefootnote}{\arabic{footnote}}



\begin{abstract}
The two-sided Rayleigh quotient iteration proposed by Ostrowski computes
a pair of corresponding left-right eigenvectors of a matrix $C$. We
propose a Grassmannian version of this iteration, i.e., its iterates are
pairs of $p$-dimensional subspaces instead of one-dimensional subspaces
in the classical case. The new iteration generically converges locally
cubically to the pairs of left-right $p$-dimensional invariant subspaces
of $C$. Moreover, Grassmannian versions of the Rayleigh quotient
iteration are given for the generalized Hermitian eigenproblem, the
Hamiltonian eigenproblem and the skew-Hamiltonian eigenproblem.
\end{abstract}

\noindent {\bf Keywords.} Block Rayleigh quotient iteration, two-sided
iteration, Grassmann manifold, generalized eigenproblem,
Hamiltonian eigenproblem.

\noindent
{\bf AMS subject classification.} 65F15

\pacomm{STILL TO DO:

Give an application related to the computation of the H-infty
norm. References: \cite{GenDooVer1998,SreDooTit1996}. The idea is
roughly as follows. We have a Hamiltonian matrix formed from the
matrices $A$, $B$, $C$ of an LTI system and a parameter
$\gamma$. Generically, for $\gamma$ smaller than and sufficiently
close to the H-infty norm of the system, the Hamiltonian has four
eigenvalues (two complex conjugate pairs) on the imaginary
axis. Compute these eigenvalues (the two upper ones on the imaginary
axis), define the new $\gamma$ as their mean, and repeat. Once the
computation has been done for the first $\gamma$, the subsequent
computation amount to refining a four-dimensional eigenspace, which
moreover is a \emph{full} eigenspace, as defined in this paper---the
spectrum related to the eigenspace is symmetric wrt the origin. This
is precisely the case that the two-sided GRQI can tackle. An instance
of such an H-infty computation problem can be found
in~\cite{SreDooTit1996}.

Even better, the two-sided GRQI also computes the eigenvectors, which
makes it possible to compute the slope of the curve
$\sigma_{\max}(G(j\omega))$ at the intersection points, which yields a
better approximation than bisection.
}


\section{Introduction}

The Rayleigh quotient iteration (RQI) is a classical method for
computing eigenvectors of a Hermitian matrix
\textcolor{black}{$A=A^H$}~\cite{Par74,Par98}. The RQI is a particular
inverse iteration~\cite{Ips97} where the shift is the Rayleigh
quotient evaluated at the current iterate. The Rayleigh quotient is an
efficient shift because in a neighborhood of any eigenvector of $A$ it
yields a quadratic approximation of the corresponding eigenvalue. This
remarkable property endows the iteration with cubic rate of convergence to
the eigenvectors of $A$ (see~\cite{Par74,Par98} or the sketch of proof
in~\cite{AMSV2002-01}). Thanks to its fast convergence, the RQI is
particularly efficient for refining estimates of eigenvectors.

In some cases, one has to refine an estimate of a $p$-dimensional
invariant subspace (or \emph{eigenspace}) of $A$. A reason
for considering an eigenspace instead of individual eigenvectors may
be that the eigenvectors themselves are ill-conditioned while the
subspace is not (see e.g.~\cite{Ste73}) or just because they are not
relevant for the application. Several methods have been proposed for
refining invariant subspace estimates. A quadratically convergent
iteration for refining eigenspaces of arbitrary (possibly
non-Hermitian) matrices was proposed by Chatelin~\cite{Cha84,Dem87}. It
uses Newton's method for solving the Riccati equation obtained by
expressing the eigenproblem in inhomogeneous coordinates. Similar
Newton-like iterations for eigenspace refinement were obtained using a
differential-geometric approach~\cite{EAS98,LE2002,AMS2004-01};
see~\cite{AMS-book-u} for an overview.  
In~\cite{Smi97,AMSV2002-01}, it was shown that the RQI, originally
defined on the set of one-dimensional subspaces of $\rr^n$, can be
generalized to operate on the set of $p$-dimensional subspaces of
$\rr^n$. The generalized iteration, called block-RQI or Grassmann-RQI
(because the set of the $p$-dimensional subspaces of $\rr^n$ is termed
a \emph{Grassmann manifold}) converges locally cubically to the
$p$-dimensional eigenspaces of $A=A^H$.

It is natural to ask whether the Grassmann-RQI method can be adapted to deal
with non-Hermitian matrices. This is the topic of the present paper.

The underlying idea comes from Ostrowski's series of papers dedicated
to the RQI~\cite{Ost58-x}. Let $C$ be a nonnormal matrix. Then the
quadratic approximation property of the Rayleigh quotient is lost (and
moreover the global convergence properties of the RQI become
weaker, see~\cite{BS90}). This drawback was avoided by
Ostrowski~\cite{Ost58-3,Par74} by considering the bilateral Rayleigh
quotient $\rho(y_L,y_R):= y_L^HCy_R/y_L^Hy_R$ which displays the
quadratic property in the neighborhood of the pairs of left-right
nondefective eigenvectors of $C$. Using this Rayleigh quotient as a
shift, he derived a two-sided iteration (see Algorithm~\ref{al:2sRQI}
below) that operates on pairs of vectors (or pairs of one-dimensional
subspaces, since the norm is irrelevant) and aims at converging to
pairs of left-right eigenvectors of $C$. The rate of convergence is in
cubic in nondegenerate cases. The possibility of solving the two-sided
RQI equations approximately was investigated in~\cite{HS2003}. 

In the present paper, we generalize Ostrowski's two-sided RQI to operate on
pairs of $p$-dimensional subspaces (instead of one-dimensional subspaces in the
original iteration). The new iteration, called \emph{Two-Sided Grassmann-RQI}
(\emph{2sGRQI}), converges locally cubically to the pairs of
left-right $p$-dimensional eigenspaces of $C$ (see
Section~\ref{sec:cubic}).
%
%
Comparison between Chatelin's iteration and the 2sGRQI
(Section~\ref{sec:comparisons}) shows that each method has its
advantages and drawbacks. Main advantages of the 2sGRQI over
Chatelin's iteration are the higher rate of convergence, the
simultaneous computation of left and right eigenspaces, and the
simpler structure of the Sylvester equations. On the other hand, the
2sGRQI does not behave satisfactorily when $C$ is defective and it
involves two Sylvester equations instead of one.
%
We also show that in some structured eigenproblems, namely
$E$-(skew-)Hermitian matrices with $E=\pm E^H$, a relation
$\calY_L=E\calY_R$ between left and right subspaces is invariant by
the 2sGRQI mapping (Section~\ref{sec:structured-eigenproblems}). In
particular, this observation yields a modified one-sided Grassmann-RQI
for the Hamiltonian eigenproblem. 
%
We report on numerical
experiments in Section~\ref{sec:numerical-experiments} and conclusions
are drawn in Section~\ref{sec:concl}.

\section{Preliminaries}
\label{sec:preliminaries}


This paper uses a few elementary concepts related to the algebraic
eigenvalue problem, such as principal vectors, Jordan blocks and
nonlinear elementary divisors. A classical reference
is~\cite{Wil65}. Notions of subspaces and distance between them
can be found in~\cite{Ste73}.

\pacomm{(I switched to the complex case, because I am
  using Ritz vectors to span subspaces, and Ritz vectors are
  complex in general. This means: symmetric becomes Hermitian, $^H$
  becomes $^H$, $^{-H}$ becomes $^{-H}$, $\rr$ becomes $\cc$, likewise
  for the Stiefel commands, ``real projective space'' becomes
  ``projective space'', the notation $\bbRP$ becomes $\bbP$.)}

The superscript $^H$ denotes the conjugate transpose. In accordance
with Parlett's conventions~\cite{Par74,Par98}, we try to reserve the
letter $A$ for Hermitian matrices while $C$ may denote any matrix.  We
use $\Grass{p}{n}$ to denote the Grassmann manifold of the
$p$-dimensional subspaces of $\cc^n$, $\bbP^{n-1}$ to denote the
projective space \textcolor{black}{(i.e., the set of all
  one-dimensional subspaces of $\cc^n$)}, and $\ST{p}{n}$ to denote
the noncompact Stiefel manifold, i.e., the set of $n$-by-$p$ matrices
with full rank. The space spanned by the columns of $Y\in\STpn$ is
denoted by $\ecl{Y}$ and called the \emph{span} of $Y$. The norm of a
vector $x$ is $\|x\|=\sqrt{x^Hx}$.  The spectral norm of a matrix $T$,
denoted by $\|T\|$, is the largest singular value of $T$. The
Hermitian angle $\angle(x,y)$ between two vectors $x$ and $y$ in
$\cc^n$ is given by $\cos\angle(x,y) =
\tfrac{|x^Hy|}{\|x\|\|y\|}$~\cite{Sch2001}. The angle between a vector
$y\in\cc^n$ and a subspace $\calX$ spanned by $X\in\STpn$ is
$\angle(X,y) = \min_{x\in\calX,\|x\|=1}\angle(x,y)$.  The angle
$\angle(X,Y)$ between two subspaces spanned by $X\in\STpn$ and
$Y\in\STpn$ is defined as the largest principal angle between the two
subspaces, given by $\cos\angle(X,Y) = \sigma_{\min}(\tilde X^H\tilde
Y)$ where $\tilde X$ and $\tilde Y$ are orthonormal bases for
$\ecl{X}$ and $\ecl{Y}$, respectively, and $\sigma_{\min}$ denotes the
smallest singular value~\cite{GalHeg2006}. The following proposition
is a generalization of~\cite[Th.~3.1]{AMSV2002-01} to the complex
case.
\begin{prpstn}  \label{prp:tan-delta}
Let $[X|X_\perp]$ be a unitary matrix of order $n$, with $X$ of
dimension $n\times p$, and let $K$ be an $(n-p)\times p$ matrix. Then 
\[
\tan\angle(X,X+X_\perp K) = \|K\|.
\]
\end{prpstn}
\begin{proof}
The matrix $\tilde Y = (X+X_\perp K)(I+K^HK)^{-1/2}$ is an orthonormal
matrix with the same span as $X+X_\perp K$. It follows that
$\cos\angle(X,X+X_\perp K) = \sigma_{\min} (X^H\tilde Y) =
\sigma_{\min} (I+K^HK)^{-1/2} = (1+\sigma_{\max}^2(K))^{-1/2} =
(1+\|K\|^2)^{-1/2}$. The conclusion follows from the trigonometric
formula $\tan^2 a = (1-\cos^2 a)/\cos^2 a$.
\end{proof}
 
We now briefly recall some
basic facts about invariant subspaces.
\begin{dfntn}[eigenspaces]  \label{def:eigenspace}
Let $\calX$ be a $p$-dimensional subspace
of $\cc^n$ and let $X=[X_1,X_2]$ be a unitary $n\times n$
matrix such that $X_1$ spans $\calX$. Then $X^HCX$ may be
partitioned in
the form $X^HCX=\left(\begin{smallmatrix} C_{11} & C_{12} \\
C_{21} & C_{22}
\end{smallmatrix}\right)$ where $C_{11}\in\cc^{p\times p}$. The
subspace $\calX$ is an \emph{eigenspace} (or \emph{invariant
  subspace}) of $C$ if $C_{21}=0$, i.e., $C\calX \subset \calX$. By
\emph{spectrum} of $\calX$, we mean the set of eigenvalues of
$C_{11}$. We say that $\calX$ is a \emph{nondefective} invariant
subspace of $C$ if $C_{11}$ is nondefective. The invariant subspace
$\calX$ is termed \emph{spectral} if $C_{11}$ and $C_{22}$ have no
eigenvalue in common~\cite{GLR86}. The eigenspaces of $C^H$ are
called \emph{left eigenspaces} of $C$. We say that $(\calY_L,\calY_R)$
is a \emph{pair of spectral left-right eigenspaces} of $C$ if
$\calY_L$ and $\calY_R$ are spectral left and right eigenspaces of $C$
with the same spectrum.
\end{dfntn}
\pacomm{(Instead of ``spectral'' invariant subspace, we could say
``simple'' invariant subspace, as Stewart does.)}
The span of $Y\in\STpn$ is an eigenspace of $C$ if and only if
there exists a matrix $M$ such that $CY=YM$. Each spectral
eigenspace is \emph{isolated}, i.e., there exists a ball
in $\Grasspn$ centered on $\calV$ that does not contain any
eigenspace of $C$ other than $\calV$. We will also need the
following result~\cite[\S 7.6.3]{GV96}, of which we give an informative proof.
\begin{prpstn}  \label{thm:S}
If $(\calY_L,\calY_R)$ is a pair of spectral left-right
eigenspaces of $C$, then there exists an invertible matrix $S$
such that the first $p$ columns of $S$ span $\calY_R$, the first
$p$ columns of $S^{-H}$ span $\calY_L$, and
$S^{-1}CS=\left[\begin{smallmatrix} D_1 & 0 \\ 0 &
D_2\end{smallmatrix}\right]$ with $D_1\in\cc^{p\times p}$.
\end{prpstn}
\begin{proof}
Let $(\calY_L,\calY_R)$ be a pair of spectral left-right eigenspaces of $C$.
Then there exists $X$ unitary such that $X^HCX=\left[\begin{smallmatrix}
    C_{11} & C_{12} \\ 0 & C_{22}
\end{smallmatrix}\right]$, the first $p$ columns of $X$ span
$\calY_R$, and $C_{11}$ and $C_{22}$ have no eigenvalue in common.
Therefore, there exists a matrix $L$ such that
$C_{11}L-LC_{22}=-C_{12}$~\cite{Ste73}. Let
$S=X\left[\begin{smallmatrix} I & L \\ 0 &
I\end{smallmatrix}\right]$. Then the first $p$ columns of $S$ span
$\calY_R$. One easily checks that
$S^{-1}CS=\left[\begin{smallmatrix} C_{11} & 0 \\ 0 & C_{22}
\end{smallmatrix}\right]$. Moreover $C^HS^{-H} =
S^{-H}\left[\begin{smallmatrix} C_{11}^H & 0 \\ 0 & C_{22}^H
\end{smallmatrix}\right]$. Therefore the first $p$ columns of
$S^{-H}$ span an eigenspace of $C^H$ whose spectrum is the same as
the one of $\calY_R$. That is, the first $p$ columns of $S^{-H}$ span
$\calY_L$.
\end{proof}

The Rayleigh quotient iteration (RQI) is a classical method for
computing a single eigenvector of a Hermitian matrix $A$. It induces
an iteration on the projective space $\bbP^ {n-1}$ that can be written
as follows.
\begin{lgrthm}[RQI on projective space]
\label{al:RQI} \textcolor{black}{Let $A=A^H$ be an $n\times n$
matrix. Given $\calS_0$ in the projective space $\bbP^{n-1}$, the RQI
algorithm produces a sequence of elements of $\bbP^{n-1}$ as follows. For
$k=0,1,2,\ldots$,}
\begin{enumerate*}
\item Pick $y$ in $\cc^n\setminus\{0\}$ such that
$\eqcllo{y}=\calS_k$.
\item Compute the Rayleigh quotient $\rho_k = (y^HAy)/(y^Hy)$.
\item If $A-\rho_k I$ is singular, then solve for its kernel and
stop. Otherwise, solve the system
\begin{equation}
(A-\rho_k I)z = y \label{eq:RQI}
\end{equation}
for $z$.
\item $\calS_{k+1}:= \eqcllo{z}$.
\end{enumerate*}
\end{lgrthm}
It is shown in~\cite{BS89} that around each (isolated) eigenvector
of $A$, there is a ball in which cubic convergence to the
eigenvector is uniform. The size of the ball depends on the
spacing between the eigenvalues. Globally, the RQI converges to an
eigenvector for any initial point outside a certain set of measure zero
described in~\cite{BS89}.


The \emph{Grassmann-Rayleigh Quotient Iteration}
(GRQI) is a generalization of the RQI that operates on $\Grass{p}{n}$,
the set of all $p$-dimensional subspaces of $\cc^n$~\cite{AMSV2002-01}.
\begin{lgrthm}[GRQI]
\label{al:GRQI}
Let $A=A^H$ be an $n\times n$
matrix. Given $\calY_0\in\Grasspn$, the GRQI algorithm
produces a sequence of $p$-dimensional subspaces of $\cc^n$ by iterating from
$\calY_0$ the mapping $\Grass{p}{n}\rightarrow\Grasspn: \calY \mapsto \calY_+$
defined as follows.
\begin{enumerate*}
\item Pick $Y \in \ST{p}{n}$ such that $\eqcllo{Y} = \calY$.
\item Solve the Sylvester equation
\begin{equation}
AZ - Z (Y^HY)^{-1} Y^HAY = Y
\label{eq:GRQI}
\end{equation}
for $Z \in \cc^{n\times p}$.
\item Define $\calY_+:= \eqcllo{Z}$.
\end{enumerate*}
\end{lgrthm}
It is shown in~\cite{AMSV2002-01} that the subspace $\calY_+$ does not depend
on the choice of basis $Y$ for $\calY$ in the first step. This iteration
converges cubically to the $p$-dimensional eigenspaces of $A$, which are the
only fixed points.

When the matrix $A$ is not normal, 
the stationary property of the Rayleigh quotient
fails. Consequently, the convergence rate of the RQI can be at best quadratic.
In order to recover cubic convergence, Ostrowski~\cite{Ost58-3} proposed a 
two-sided version of the RQI, \textcolor{black}{formulated as follows
in~\cite{Par74}.} 
\begin{lgrthm}[Two-Sided RQI]
\label{al:2sRQI} Let $C$ be an $n$-by-$n$ matrix. Pick initial
vectors $v_0$ and $u_0$ satisfying $v_0^Hu_0 \neq 0$,
$\|v_0\|=\|u_0\|=1$. For $k=0,1,2,\ldots$,
\\ 1.~~Compute $\rho_k = v_k^HCu_k/v_k^Hu_k$. \\ 2.~~If $C-\rho_kI$ is
singular solve $y^H(C-\rho_kI)=0$ and $(C-\rho_kI)x=0$ for $y,x\neq 0$
and stop, otherwise \\ 3.~~Solve both $v_{k+1}^H(C-\rho_kI)=v_k^H\nu_k$,
$(C-\rho_kI)u_{k+1}=u_k\tau_k$, where $\nu_k$ and $\tau_k$ are
normalizing factors. \\ 4.~~If $v_{k+1}^Hu_{k+1}=0$, then stop and admit
failure.
\end{lgrthm}
The Two-Sided RQI converges with cubic rate to the pairs of left-right
eigenvectors of $C$ with linear elementary divisor~\cite{Par74}.

\section{Two-Sided GRQI}
We propose the following generalization of
the Two-Sided RQI, which we call the \emph{Two-Sided Grassmann-Rayleigh
Quotient Iteration (2sGRQI)}.
\begin{lgrthm}[2sGRQI]
\label{al:2sGRQI}
\textcolor{black}{Let $C$ be an $n$-by-$n$ matrix. Given
$(\calY_{L_0},\calY_{R_0})\in\Grasspn\times\Grasspn$, the 2sGRQI algorithm
produces a sequence of pairs of $p$-dimensional subspaces of $\cc^n$ by
iterating from $(\calY_{L_0},\calY_{R_0})$ the mapping
$(\calY_L,\calY_R)\mapsto (\calY_{L+},\calY_{R+})$ defined as follows.} \\
1.~~Pick $Y_L$ and $Y_R$ in $\ST{p}{n}$ such that $\eqcllo{Y_L} =
\calY_L$ and $\eqcllo{Y_R} = \calY_R$. \\ 2.~~Solve the Sylvester equations
\begin{subequations} \label{eq:GRQI-LR}
\begin{gather}
CZ_R - Z_R \underbrace{(Y_L^HY_R)^{-1} Y_L^HCY_R}_{R_R} = Y_R
\label{eq:GRQI-R} \\
Z_L^H C - \underbrace{ Y_L^HCY_R (Y_L^HY_R)^{-1} }_{R_L} Z_L^H = Y_L^H
\label{eq:GRQI-L}
\end{gather}
\end{subequations}
for $Z_L$ and $Z_R$ in $\cc^{n\times p}$. \\ 3.~~Define
$\calY_{L+}:= \eqcllo{Z_L}$ and $\calY_{R+}:= \eqcllo{Z_R}$. \\
\end{lgrthm}


In point 1, one has to choose bases for $\calY_L$ and $\calY_R$. There
are infinitely many possibilities. Indeed, if $Y$ is a basis of
$\calY$, then $\{YM:M\in\ST{p}{p}\}$ is the (infinite) set of all
bases of $\calY$. Therefore, one has to make sure that $\calY_{L+}$
and $\calY_{R+}$ do not depend on the choice of basis.
\textcolor{black}{By a straightforward adaptation of the development
  carried out in~\cite{AMSV2002-01} for the GRQI algorithm}, if
$(Y_L,Y_R,Z_L,Z_R)$ solve~\eqref{eq:GRQI-LR} then
$(Y_LM,Y_RN,Z_LM,Z_RN)$ also solve~\eqref{eq:GRQI-LR} for all $M$, $N$
in $\ST{p}{p}$. Hence, the spans of $Z_L$ and $Z_R$ only depend on
$\calY_L$ and $\calY_R$, and not on the choice of the bases $Y_L$ and
$Y_R$.

In point 2, the matrix $Y_L^HY_R$ may not be invertible. This
corresponds to point 4 in the Two-Sided RQI
(Algorithm~\ref{al:2sRQI}). However, if $(\calY_L,\calY_R)$ is a pair of
spectral left-right eigenspaces of $C$, then $Y_L^HY_R$ is invertible
\textcolor{black}{(as a consequence of Proposition~\ref{thm:S})}, and by
continuity invertibility holds on a neighborhood of the pair of
eigenspaces. 

In point 2, the (uncoupled) Sylvester equations~\eqref{eq:GRQI-LR} may
fail to admit one and only one solution. This situation happens if and
only if $(Y_R,Y_L)$ belongs to the set
\begin{eqnarray*}
\calS &:=& \{(Y_L,Y_R) \in \STpn\times\STpn: R_R \text{ exists and
}
\sigma(C)\cap\sigma(R_R)\neq\emptyset \} \\
&=& \bigcup_{\lambda\in\sigma(C)} \{(Y_L,Y_R) \in
\STpn\times\STpn: R_R \text{ exists and } \det(R_R-\lambda I)=0\};
\end{eqnarray*}
this follows directly from the characterization of the eigenvalues of
Sylvester operators~\cite[Th.~4.4]{Ste73}. Since $\calS$ is the finite
union of algebraic sets, it has measure zero and the interior of its
closure is empty. This means that if $(\hat Y_L,\hat Y_R)$ does not
yield a unique solution, then there exists, arbitrarily close to
$(\hat Y_L,\hat Y_R)$, a pair $(Y_L,Y_R)$ and a neighborhood of this
pair on which the solution $(Z_L,Z_R)$ of~\eqref{eq:GRQI-LR} exists
and is unique. In our numerical experiments, when such a singularity
occurs (i.e., when the solution of the Sylvester equations returned by
Matlab contains {\tt Inf}'s or {\tt NaN}'s), we slightly perturb the
system. A justification for this technique is given
in~\cite{AMSV2002-01} and the numerical tests performed in
Section~\ref{sec:numerical-experiments} illustrate that the technique
works well in practice. 
\pacomm{(This discussion makes the
  paper vulnerable, I think, but we have to talk about it.)}

In point 3, if $Z_L$ or $Z_R$ is not full rank, then
$(\calY_{L+},\calY_{R+})$ does not belong to
$\Grasspn\times\Grasspn$. A tall $n\times p$ matrix $Z$ is rank
deficient if and only if all its $p\times p$ minors are zero.
Therefore, the set
\[
\calD:= \{(Y_L,Y_R): \rank(Z_L) < p \text{ or } \rank(Z_R) < p\}
\]
is a subset of a finite union of algebraic sets. So here again, $Z_L$ and $Z_R$
are full rank for a generic choice of $Y_L$, $Y_R$.


In practice, only a few iterates will be computed. In finite precision
arithmetic, the iterates no longer improve after a few (typically
two or three) iterations because of numerical errors (see numerical
experiments in Section~\ref{sec:numerical-experiments}). Stopping
criteria can rely on the principal angles between two successive
iterates and on the principal angles between $\calY_R$ and $A\calY_R$ or
$\calY_L$ and $A^H\calY_L$.

\section{Practical implementation}
\label{sec:practical-implementation}

The major computational task in both GRQI (Algorithm~\ref{al:GRQI})
and 2sGRQI (Algorithm~\ref{al:2sGRQI}) is to solve the Sylvester
equations. For GRQI, it is recommended to choose an orthonormal basis
$Y$ (i.e., $Y^HY=I_p$) that makes $Y^HAY$ diagonal. This requires
solving a $p$-dimensional eigenproblem, which is cheap when $p$ is
small. With $Y^HY=I$ and $Y^HAY$ diagonal, the GRQI
equation~\eqref{eq:GRQI} decouples into $p$ linear systems for the $p$
columns of $Z$. We refer to~\cite{AMSV2002-01} for details.

The case of 2sGRQI (Algorithm~\ref{al:2sGRQI}) is quite different. The
matrices $R_R$ and $R_L$ in the 2sGRQI equations~\eqref{eq:GRQI-LR}
are not Hermitian and they may not be diagonalizable. A possible
approach to solving an equation such as~\eqref{eq:GRQI-R} is to reduce
it to a certain triangular structure by means of unitary
transformations and solve the new system of equations using back
substitution, as described in~\cite{GLAM92}. However, we observed in
numerical experiments that this technique tends to yield rather
inaccurate results when the iterates get close to a solution (the
final error was sometimes around $10^{-11}$ whereas the machine
epsilon was approximately $2.2\cdot 10^{-16}$, to be compared with the
results in Table~\ref{tab:conv}). The reason seems to lie in the fact
that the norm of the solution $z$ to the equation $(C-\rho I)z = x$
becomes very sentitive to $\rho$ when $\rho$ gets close to an
eigenvalue of $C$. The magic of the classical RQI is that the
direction of $z$ is well-conditioned, as pointed out by Peters and
Wilkinson~\cite{PW79}. However, in the block case, the magic weakens
because, in view of the workings of back substitution, the large
numerical error in the norm of any one column of $Z$ jeopardizes the
accuracy of the other columns whose computation depends on that
column.

Consequently, we recommend reducing the small block shifts $R_R$ and
$R_L$ in the 2sGRQI equations~\eqref{eq:GRQI-LR} to quasi-diagonal
form, or to (complex) diagonal form if complex arithmetic is
available. To fix ideas, assume complex arithmetic and consider the
first equation,~\eqref{eq:GRQI-R}, namely
\[
C Z_R - Z_R R_R = Y_R.
\]
Assuming that $R_R$ is nondefective, let $R_R = W_R \,
\diag(\rho_1,\ldots,\rho_p)\, W_R^{-1}$ be an eigenvalue decomposition
of $R_R$. Multiplying~\eqref{eq:GRQI-R} on the right by $W_R$ yields
\[
C \tilde Z_R - \tilde Z_R \diag(\rho_1,\ldots,rho_p) = \tilde Y_R
\]
where $\tilde Y_R = Y_R W_R$ and $\tilde Z_R = Z_R W_R$. The advantage
of this reformulation of~\eqref{eq:GRQI-R} is that it yields $p$
decoupled equations
\[
(C-\rho_i I) \tilde Z_R e_i = \tilde Y_r e_i
\]
for each column $\tilde Z_R e_i$ of $\tilde Z_R$. Back propagation is
thus no longer needed. A drawback is that this technique does not work
when $R_R$ is defective, and numerical errors on $W_R$ may become
large when $R_R$ is close to being defective. Nevertheless, in our
extensive numerical experiments on randomly chosen matrices, these
difficulties were not noticed (see
Section~\ref{sec:numerical-experiments}). 

The same kind of discussion applies to the left
equation~\eqref{eq:GRQI-L}. Note that since $R_L = (Y_L^HY_R) R_R
(Y_L^HY_R)^{-1}$ is a similarity transformation of $R_R$, we have that
$W_L = (Y_L^HY_R)W_R$ is a matrix of eigenvector of $R_L$. Hence, the
eigendecomposition of $R_L$ is readily obtained from that of $R_R$.

\section{Local convergence}
\label{sec:cubic}

The following local convergence analysis can be thought of as a
two-sided generalization of the proof of cubic convergence of the
block-RQI (equivalent to the Grassmann-RQI of~\cite{AMSV2002-01})
given in~\cite{Smi97}. 

Let $(\calV_L,\calV_R)$ be a pair of spectral left-right eigenspaces
of $C$, and let $V_L$ and $V_R$ be corresponding eigenbases. We assume
that the eigenspaces are nondefective, that is, the matrix
$(V_L^HV_R)^{-1}V_L^HCV_R$ is diagonalizable by a similarity
transformation. Since $(\calV_l,\calV_R)$ is nondefective, it follows
that for all $\calY_L$ and $\calY_R$ sufficiently close to $\calV_L$
and $\calV_R$, the block Rayleigh quotients $R_R$ and $R_L$ are
diagonalizable by similarity transformations $W_R$ and $W_L$.
Equations~\eqref{eq:GRQI-LR} thus can be solved in two steps: (i)
diagonalize the small block Rayleigh quotients, hence decoupling the
equations and reducing them to classical two-sided RQI equations; (ii)
solve the decoupled two-sided RQI equations, yielding matrices
$Z_LW_L$ and $Z_RW_R$ that span $\calY_{L+}$ and $\calY_{R+}$. The key
of the convergence analysis is an ``oblique'' generalization
of~\cite[Th.~2]{Ste2001-LAA}, showing that the angles between the
right Ritz vectors (the columns of $Y_RW_R$) and the ``corresponding''
right eigenvectors of $A$ are of the order of the largest principal
angle between $\calY_R$ and $\calV_R$, and likewise for the left Ritz
vectors and eigenvectors; see Lemma~\ref{lmm:Ste01-obl}. Then the
result follows quite directly from the cubic convergence of the
non-block two-sided RQI.

\begin{lmm} \label{lmm:Ste01-obl} 
Let $(\lambda,x)$ be an eigenpair of an $n\times n$ matrix $C$. Let
$Y_R$ and $Y_L$ be orthonormal $n\times p$ matrices, $p<n$, such that
$Y_L^HY_R$ is invertible. Let $w_R$ be an eigenvector of 
\[
B:= (Y_L^HY_R)^{-1} Y_L^HCY_R
\]
associated with the eigenvalue of $B$ that is closest to
$\lambda$. Then
\[
\boxed{
\sin\angle(Y_Rw_R,x) \leq \left[
  1+\frac{2(\cos\delta)^{-1}r_L\alpha_\delta(\epsilon)}{\sep(w_R^HBw_R,w_{R\perp}^HBw_{R_\perp})-r_L\gamma_\delta(\epsilon)}\right] (1+\tan\delta) \epsilon
}
\]
where $\epsilon := \sin\angle(Y_R,x)$ is the angle between the
direction of $x$ and the span of $Y_R$, $\delta := \angle(Y_R,Y_L)$ is
the largest principal angle between the spans of $Y_R$ and $Y_L$,
$\alpha_\delta(\epsilon) :=
\tfrac{1}{\sqrt{1-\epsilon^2}-\epsilon\tan\delta}$ satisfies
$\lim_{\epsilon\to0}\alpha_\delta(\epsilon) = 1$,
$\gamma_\delta(\epsilon):= (\cos\delta
(\sqrt{1-\epsilon^2}-\epsilon\tan\delta))^{-1} (1+\tan\delta)\epsilon$
satisfies $\lim_{\epsilon\to0}\gamma_\delta(\epsilon) = 0$, and $r_L
:= \|Y_{L\perp}^HA^HY_L\|$ where $Y_{L\perp}\in\cc^{n\times(n-p)}$ is
an orthonormal basis of the orthogonal complement of the span of
$Y_L$. 
\end{lmm}
\begin{proof}
  It is readily checked that the statement is not affected by a
  unitary change of coordinates in $\cc^n$. Therefore, without loss of
  generality, we
  work in a unitary coordinate system such that $Y_R = \begin{bmatrix} I_p \\
    0_{(n-p)\times p}
  \end{bmatrix}$. Let $Y_{L\perp}\in\cc^{n\times(n-p)}$ and
  $Y_{R\perp}\in\cc^{n\times(n-p)}$ be orthonormal bases of the
  orthogonal complements of the spans of $Y_L$ and $Y_R$,
  respectively. Assume without loss of generality that the eigenvector
  $x$ has unit norm.  Consider the block decompositions
  $x=\begin{bmatrix}x_a \\ x_b
  \end{bmatrix}$ and $Y_L = \begin{bmatrix} Y_{La} \\ Y_{Lb}
  \end{bmatrix}$. Consider also the decomposition $x=Y_R x_R +
  Y_{L\perp}x_{L\perp}$, which yields
\[
x_R:=(Y_L^HY_R)^{-1}Y_L^H x, \qquad
x_{L\perp}:=(Y_{R\perp}^HY_{L\perp})^{-1} Y_{R\perp}^H x.
\]
Since $\epsilon=\sin\angle(Y_R,x)$, we have 
$\|x_a\|^2 = 1-\epsilon^2$ and $\|x_b\| = \epsilon$. We also have
$(Y_L^HY_R)^{-1}Y_L^H = \begin{bmatrix}I & T \end{bmatrix}$ where $T=
(Y_{La})^{-1}Y_{Lb}$. It follows from Proposition~\ref{prp:tan-delta}
that $\|T\|=\tan\delta$. We also
obtain
\[
Y_Rx_R = \begin{bmatrix}I \\ 0 \end{bmatrix} \begin{bmatrix} I & T
\end{bmatrix} x = 
\begin{bmatrix} x_a + Tx_b \\ 0 \end{bmatrix}.
\]
Acceptable choices for $Y_{L\perp}$ and $Y_{R\perp}$ 
are $Y_{L\perp} = \begin{bmatrix} -T \\
  I_{n-p} \end{bmatrix} (I_{n-p} + T^H T)^{-1/2}$ and $Y_{R\perp} =
\begin{bmatrix} 0_{p\times(n-p)} \\ I_{n-p} \end{bmatrix}$. This yields
$x_{L\perp} =  (I_{n-p}+T^HT)^{1/2} x_b$ and thus $\|x_L\| \leq
\sqrt{1+\tan^2\delta}\,\epsilon$. 

Since $\sin\angle(u,v) \leq \sin\angle(u,w) + \sin\angle(w,v)$ for all
$u,v,w\in\cc^n_0$, we have
\begin{equation}  \label{eq:a+a}
\angle(Y_Rw_R,x) \leq \angle(Y_Rw_R, Y_Rx_R) + \angle(Y_R x_R,x).
\end{equation}

Let us first consider the second term in~\eqref{eq:a+a}. 
Since 
\[
\sin\angle(Y_Rx_R,x) \leq \|Y_Rx_R-x\| \leq \| Y_Rx_R- \begin{bmatrix} x_a \\
0 \end{bmatrix}\| + \| \begin{bmatrix} x_a \\
0 \end{bmatrix} - x\|,
\]
it follows that
\begin{equation}  \label{eq:a2}
\sin\angle(Y_R x_R,x) \leq \|Tx_b\| + \|x_b\| 
\leq \tan\delta \ \epsilon + \epsilon = (1+\tan\delta)\epsilon.
\end{equation}
Note also for later use that, for all small $\epsilon$ such that
$\sqrt{1-\epsilon^2}>\epsilon\tan\delta$, we also obtain that $\|x_R\| \geq
\left|\|x_a\| - \|Tx_b\|\right| \geq
\sqrt{1-\epsilon^2} - \epsilon \tan\delta$. 

We now tackle the first term in~\eqref{eq:a+a}. Since $Y_R$ is
orthonormal, it follows that $\angle(Y_Rw_R,Y_Rx_R) =
\angle(w_R,x_R)$. Pre-multiplying the equation $Cx=\lambda x$ by
$(Y_L^HY_R)^{-1}Y_L^H$ yields 
\[
(Y_L^HY_R)^{-1}Y_L^H C (Y_Rx_R + Y_{L\perp}x_{L\perp}) = x_R \lambda,
\]
which can be rewritten as
\[
(B+E)\hat x_R = \lambda \hat x_R,
\]
where $\hat x_R := x_R \|x_R\|^{-1}$ and
\[
E := (Y_L^HY_R)^{-1}Y_L^HAY_{L\perp}
x_{L\perp} \|x_R\|^{-1} \hat x_R^H.
\]
Then, by~\cite[Th.~5.1]{JiaSte2000},
\[
\sin \angle (w_R,\hat x_R) \leq \tan \angle (w_R,\hat x_R) \leq
\frac{2\|E\|}{\sep\left(w_R^HBw_R,(w_R)_\perp^HB(w_R)_\perp\right) -
2\|E\|} 
\]
if the bound is smaller than $1$. 
The expression of the bound can be simplified using
\begin{multline*}
\|E\| = \|(Y_L^HY_R)^{-1}Y_L^HCY_{L\perp}
x_{L\perp} \| \|x_R\|^{-1} 
\leq
\|(Y_L^HY_R)^{-1}\| \|Y_{L\perp}^HC^HY_L\| \|x_{L_\perp}\| \|x_R\|^{-1}
\\ \leq \frac{1}{\cos\delta} r_L (1+\tan\delta)\epsilon
\frac{1}{\sqrt{1-\epsilon^2}-\epsilon\tan\delta},
\end{multline*}
where we have used the bound $\sqrt{1+\tan^2\delta} \leq
(1+\tan\delta)$ that holds for all $\delta\in[0,\tfrac{\pi}{2})$. 
Replacing all these results in~\eqref{eq:a+a} yields the desired
bound.
\end{proof}

\begin{thrm}
  \label{thm:cubic-2sGRQI-S} Let $(\calV_L,\calV_R)$ be a pair of
  $p$-dimensional spectral nondefective left-right eigenspaces of an
  $n\times n$ matrix $C$ (Definition~\ref{def:eigenspace}). Then there
  is a neighborhood $\calN$ of $(\calV_L,\calV_R)$ in
  $\Grasspn\times\Grasspn$ and a $c>0$ such that, for all
  $(\calY_L,\calY_R)\in\calN$, the subspaces $\calY_{L+}$ and
  $\calY_{R_+}$ produced by the 2sGRQI mapping
  (Algorithm~\ref{al:2sGRQI}) satisfy
\[
\angle(\calY_{L_+},\calV_L) +
  \angle(\calY_{R_+},\calV_R) \leq c \left( \angle(\calY_{L},\calV_L) +
  \angle(\calY_{R},\calV_R) \right)^3.
\]
\end{thrm}
\begin{proof}
  Since the pair of eigenspaces is assumed to be spectral, it follows
  that $\angle(V_L,V_R) < \pi/2$. Therefore, taking the neighborhood
  $\calN$ sufficiently small, one has $\angle(Y_{R},Y_{L}) \leq
  \delta^\prime < \pi/2$.  Moreover, since the pair of eigenspaces is
  assumed to be nondefective, it follows that the eigenbases $V_R$ and
  $V_L$ have full rank. Note that for each column $x$ of $V_R$, we
  have $\angle(Y_R,x)\leq\angle(Y_R,V_R)$. Lemma~\ref{lmm:Ste01-obl}
  implies that for any $c_1>1+\tan\delta^\prime$, there exists an
  $\epsilon>0$ such that, for all $(\calY_L,\calY_R)$ with
  $\angle(\calY_{L},\calV_L) + \angle(\calY_{R},\calV_R)<\epsilon$,
  the angle $\angle(Y_Rw_R,x)$ between $x$ and the nearest Ritz vector
  $Y_Rw_R$ satisfies $\angle(Y_Rw_R,x)\leq c_1\angle(Y_R,x)\leq
  c_1\angle(Y_R,V_R)$. Next, represent the subspaces
  $\calY_L$ and $\calY_R$ by their Ritz vectors, which 
  decouples~\eqref{eq:GRQI-R} into $p$ two-sided RQI equations. By
  taking $\epsilon$ sufficiently small, it follows from the cubic
  convergence of the two-sided RQI that there exists $c_2>0$ such
  that, for each column $x$ of $V_R$, we have $\angle((z_R)_i,x) < c_2
  (c_1\angle(Y_R,V_R))^3$ for at least one column $(z_R)_i$ of $Z_R$. It
  follows that $\angle(Z_R,V_R) < c_3c_2(c_1\angle(Y_R,V_R))^3$
  where $c_3$ is a constant that depends on the conditioning of the
  basis $V_R$. A similar reasoning applies to the left subspace.
\end{proof}

\section{Comparisons with Newton-based approaches}
\label{sec:comparisons}

It has been long known (see, e.g., Peters and Wilkinson~\cite{PW79})
that the RQI can be viewed as a Newton iteration. In fact, the RQI \emph{is} a
Newton method in a certain differential-geometric
sense~\cite{AMS-book-u}. However, the strict interpretation of RQI as
a Newton method disappears in the block case, as pointed out
in~\cite{AMSV2002-01}, so much so that the Grassmann-RQI can be
considered as distinct from the Newton approach.

Several Newton-based approaches for the general (non-Hermitian)
eigenvalue problem have been proposed in the litterature. In
particular, the well-known Jacobi-Davidson approach can be viewed as a
Newton method within a sequential subspace
algorithm; see, e.g.,~\cite{LE2002,AMS-book-u}. 
Here we discuss specifically the Newton method proposed by
Chatelin~\cite{Cha84} for refining eigenspace estimates. The reasoning
can be explained as follows. An $n\times p$ matrix $Y$ spans an
eigenspace of $C$ if and only if there exists a $p\times p$ matrix $M$
such that
\begin{equation}  \label{eq:CY=YM}
CY=YM.
\end{equation}
However, any subspace admits infinitely many bases, and the solutions
$Y$ of~\eqref{eq:CY=YM} are thus not isolated. A way to remove the
freedom in the choice of basis is to impose on $Y$ a normalization
condition $W^HY=I$ where $W$ is a given full-rank $n\times p$ matrix.
Then~\eqref{eq:CY=YM} becomes
\begin{equation}  \label{eq:F(Y)=0}
F(Y) := CY-Y(W^HCY)=0
\end{equation}
where the unknown $Y$ is normalized by $W^HY=I$. The Newton
iteration for solving~\eqref{eq:F(Y)=0} is given by
\begin{gather}
(I-YW^H)C\Delta - \Delta (W^HCY) = -F(Y),\ W^H\Delta = 0  \label{eq:Newton-eq} \\
Y_+ := Y+\Delta.  \label{eq:pj-update}
\end{gather}

If the basis $Y$ is chosen orthonormal and $W:=Y$,
then~\eqref{eq:Newton-eq} becomes
\begin{equation}  \label{eq:Newton-eq-2}
\Pi C \Pi \Delta - \Delta (Y^HCY) = -\Pi CY,\ Y^H\Delta = 0
\end{equation}
where $\Pi:=I-YY^H$. The resulting algorithm admits an interpretation
as a Newton method on the Grassmann manifold~\cite{AMS-book-u}. 
The rate of convergence is quadratic
in general (cubic when $C$ is Hermitian).

The constraint $Y^H\Delta=0$ can be addressed by setting
$\Delta=Y_\perp K$, where $Y_\perp$ is an orthonormal matrix with
$Y^HY_\perp=0$ and $K$ is an $(n-p)\times p$ matrix; see,
e.g.,~\cite{Dem87}. Then $Y^H\Delta=0$ is trivially satisfied and
equation~\eqref{eq:Newton-eq-2} becomes
\begin{equation}  \label{eq:Demmel-Sylvester}
(Y_\perp^HCY_\perp) K - K (Y^HCY) = -Y_\perp^HCY,
\end{equation}
i.e., a Sylvester equation without constraints on the unknown $K$.  As
pointed out in~\cite{AMSV2002-01}, solving~\eqref{eq:Demmel-Sylvester}
takes $O(n^3)$ operations even when $C$ is condensed (e.g.\
tridiagonal) because $Y_\perp^HCY_\perp$ is a large dense
$(n-p)\times(n-p)$ matrix. However, Lundstr\"om and Eld\'en proposed
an algorithm~\cite[alg.~2]{LE2002} for solving~\eqref{eq:Newton-eq-2}
that does not require the computation of $Y_\perp^HCY_\perp$. It takes
$O(np^2)$ operations to solve~\eqref{eq:Newton-eq-2} when $C$ is block
diagonal of sufficiently moderate block size and $O(n^2p)$ when $C$ is
Hessenberg. The complexity of the 2sGRQI method
(Algorithm~\ref{al:2sGRQI}) is of the same order.


A theoretical comparison between algorithms based on inverse iteration
and on Newton does not reveal that one approach has a clear edge over
the other. Among the advantages of the 2sGRQI method
(Algorithm~\ref{al:2sGRQI}) over Chatelin's method, one can mention that
the convergence of 2sGRQI is cubic instead of quadratic, and that a
pair of left-right eigenspaces is computed instead of just a
right-eigenspace.
On the other hand, Chatelin's method admits a convergence analysis
when the target eigenspace is defective~\cite{Dem87,AMS2004-01}, and
it requires solving only one Sylvester equation instead of two in
2sGRQI. However, we show in
Section~\ref{sec:structured-eigenproblems} that one Sylvester equation
suffices for 2sGRQI on some important structured eigenproblems.


\section{Structured eigenproblems}
\label{sec:structured-eigenproblems}
In this section, we show that the 2sGRQI induces particular
one-sided formulations for some structured eigenproblems.

\subsection{$E$-Hermitian eigenproblem}

Let $C$ be an $n\times n$ matrix. 
\pacomm{(Removed: not sure it is correct when $C$ is defective)
Since $C$ and $C^H$ have the
same spectrum, they are similar, i.e., there exists an invertible
matrix $E$ such that
\[
EC=C^HE,
\]
}
If there exists an invertible matrix $E$ such that 
\begin{equation}  \label{eq:EC=CTE}
EC=C^HE,
\end{equation}
then we say that $C$ is \emph{$E$-Hermitian}. If $C$ is
$E$-Hermitian, then its left and right eigenspaces are related by
the action of $E$. Indeed, let $S$ be a (complex) matrix of
principal vectors of $C$, i.e.,
\[
CS=SD
\]
where $D$ is a (complex) Jordan matrix; then,
from~\eqref{eq:EC=CTE}, one obtains $C^H(ES)=(ES)D$.

The case where $E$ is Hermitian or skew-Hermitian, i.e., $E^H=\pm E$,
is of particular interest because, as we show in the next proposition,
the relation $\calY_L=E\calY_R$ is invariant under the 2sGRQI
(Algorithm~\ref{al:2sGRQI}). Therefore, if $\calY_L=E\calY_R$, it is
not necessary to solve both~\eqref{eq:GRQI-R} and~\eqref{eq:GRQI-L}:
just solve~\eqref{eq:GRQI-R} to get $\calY_{R+}$, and obtain
$\calY_{L+}$ as $\calY_{L+}:=E\calY_{R+}$. Moreover, since the pairs
of left-right eigenspaces of $C$ also satisfy $\calV_L = E\calV_R$,
Theorem~\ref{thm:cubic-2sGRQI-S} also applies.

\begin{prpstn} \label{thm:GRQI-Esym}
Let $E$ be invertible with $E^H=\pm E$ and let $C$ be $E$-Hermitian, i.e.,
$EC=C^HE$. If $Y_L=EY$, $Y_R=Y$, and $Z$ satisfies
\begin{equation} \label{eq:GRQI-Esym}
\boxed{CZ - Z\,(Y^HEY)^{-1}(Y^HECY) = Y,}
\end{equation}
then $Z_L=EZ$ and $Z_R=Z$ satisfy the 2sGRQI
equations~\eqref{eq:GRQI-LR}. Hence, if $\calY_L=E\calY_R$, then
$\calY_{L+} = E\calY_{R+}$. Moreover, the subspace iteration
$\eqcllo{Y}\mapsto \eqcllo{Z}$ defined by~\eqref{eq:GRQI-Esym}
converges locally cubically to the spectral nondefective
right-eigenspaces of $C$.
\end{prpstn}
\begin{proof}
It is easy to check that replacing $Y_R:=Y$, $Z_R:=Z$,
$Y_L:=EY_R$, $Z_L:=EZ_R$ in~\eqref{eq:GRQI-R}
and~\eqref{eq:GRQI-L} yields~\eqref{eq:GRQI-Esym} in both cases.
In order to prove cubic convergence, it is sufficient to notice
that the pairs $(\calV_L,\calV_R)$ of eigenspaces satisfy
$\calV_L=E\calV_R$, as was shown above. Therefore, if $\calY$ is
close to $\calV_R$, then the pair
$(\calY_L,\calY_R):=(E\calY,\calY)$ is close to $(\calV_L,
\calV_R)$ and local cubic convergence to $\calV_R$ follows from
Theorem~\ref{thm:cubic-2sGRQI-S}.
\end{proof}

The discussion in Section~\ref{sec:practical-implementation} on
solving Sylvester equations applies likewise
to~\eqref{eq:GRQI-Esym}.

\subsubsection{Generalized Hermitian eigenproblem}

Using Proposition~\ref{thm:GRQI-Esym}, we show that the 2sGRQI
yields a Grassmannian RQI for the Hermitian generalized eigenproblem
$A\calV \subset B\calV$ which does not involve an explicit computation
of $B^{-1}A$.
Let $A$ and $B$ be two Hermitian $n$-by-$n$ matrices with $B$
invertible. Consider the problem of finding a $p$-dimensional
subspace $\calV$ such that $A\calV \subset B\calV$. Let
$V\in\cc^{n\times p}$ be a basis for $\calV$, then $A\calV \subset
B\calV$ if and only if there is a matrix $M$ such that $AV=BVM$.
Equivalently, $V$ spans a right-eigenspace of $B^{-1}A$, i.e.,
\[
B^{-1}AV = VM.
\]
The problem is thus to find a right-eigenspace of $C:=B^{-1} A$. The
conditions in Proposition~\ref{thm:GRQI-Esym} are satisfied with $E:=B$. The
modified GRQI equation~\eqref{eq:GRQI-Esym} becomes
\begin{equation} \label{eq:GRQI-gen}
\boxed{
AZ - BZ\,(Y^HBY)^{-1}(Y^HAY) = BY}
\end{equation}
and the subspace iteration $\eqcllo{Y}\mapsto \eqcllo{Z}$ converges
locally cubically to the spectral nondefective
eigenspaces of $B^{-1}A$. In particular, $B^{-1}A$ is nondefective
when $A$ or $B$ is positive definite.

\subsubsection{Skew-Hamiltonian eigenproblem}

Let $T$ be a skew-Hamiltonian matrix, i.e., $(TJ)^H =
-TJ$, where $J=\left(\begin{smallmatrix} 0 & I \\ -I & 0
\end{smallmatrix}\right)$, see e.g.~\cite{BBMX2002}. Equivalently,
$JT=T^HJ$, i.e., $T$ is $J$-Hermitian. Conditions in
Proposition~\ref{thm:GRQI-Esym} are satisfied with $C:=T$ and $E:=J$. The
modified GRQI equation~\eqref{eq:GRQI-Esym} becomes
\begin{equation} \label{eq:GRQI-skew-Ham}
\boxed{TZ - Z\,(Y^HJY)^{-1} (Y^HJTY) = Y}
\end{equation}
and the subspace iteration $\eqcllo{Y}\mapsto \eqcllo{Z}$ converges
locally cubically to the spectral nondefective
right-eigenspaces of $T$.

\subsection{$E$-skew-Hermitian eigenproblem}
\label{sec:EC=-CTE}

Let $E$ be an invertible $n\times n$ matrix and let $C$ be an
\emph{$E$-skew-Hermitian} $n\times n$ matrix, namely
\begin{equation} \label{eq:EC=-CHE}
EC=-C^HE.
\end{equation}
We saw in the previous section that the corresponding left and right
eigenspaces of $E$-Hermitian matrices are related by a multiplication by
$E$. The case of $E$-skew-Hermitian matrices is slightly different. 
\begin{prpstn}
  Let $C$ be an $E$-skew-Hermitian matrix. Then the spectrum of $C$ is
  symmetric with respect to the imaginary axis. In other words, if
  $\lambda$ is an eigenvalue of $C$, then so is $-\overline{\lambda}$.
  Moreover, if $\calV_L$ and $\calV_R$ are left and right eigenspaces
  of $C$ whose spectra are the symmetric image one of the other
  with respect to the imaginary axis, then $\calV_L = E\calV_R$.
\end{prpstn}
\begin{proof}
Letting $S$ be an invertible matrix of principal vectors of $C$, i.e.,
\begin{equation} \label{eq:CS=SD}
CS=SD
\end{equation}
where $D$ is a Jordan matrix,~\eqref{eq:EC=-CHE} yields
\begin{equation} \label{eq:CHES=-ESD}
C^HES = ES(-D).
\end{equation}
Hence, the matrix $-D$ is a Jordan matrix of $C^H$. Therefore, if
$\lambda$ is an eigenvalue of $C$, then $-\lambda$ is an eigenvalue of
$C^H$, and thus $-\overline{\lambda}$ is an eigenvalue of $C$. Moreover,
equations~\eqref{eq:CS=SD} and~\eqref{eq:CHES=-ESD} show that if
$\calV$ is a right-eigenspace of $C$ with eigenvalues
$\lambda_{i_1},\ldots,\lambda_{i_p}$, then $E\calV$ is a
left-eigenspace of $C$ with eigenvalues
$-\overline{\lambda}_{i_1},\ldots,-\overline{\lambda}_{i_p}$. 
\end{proof}

Consequently, letting $\calV$ be a spectral right-eigenspace of $C$,
we have that $(E\calV,\calV)$ forms a pair of spectral left-right
eigenspaces of $C$ if and only if the spectrum of $\calV$ is symmetric
with respect to the imaginary axis. We call such an
invariant subspace $\calV$ a \emph{full eigenspace} of the
$E$-skew-Hermitian matrix $C$.

If $E$ is Hermitian or skew-Hermitian, then the relation
$\calY_L=E\calY_R$ is invariant by the 2sGRQI
(Algorithm~\ref{al:2sGRQI}), as we show in the forthcoming
proposition. Therefore, if $\calY_L=E\calY_R$, it is sufficient to
solve~\eqref{eq:GRQI-R} only, and then compute
$\calY_{L+}:=E\calY_{R+}$. Moreover, the 2sGRQI iteration restricted
to the pairs $(\calY_L,\calY_R)=(E\calY,\calY)$ converges locally
cubically to the full nondefective eigenspaces of $C$.

\begin{prpstn}  \label{thm:GRQI-Eskew}
Let $E$ be invertible with $E^H=\pm E$ and let $C$ be $E$-skew-Hermitian, i.e.,
$EC=-C^HE$. If $Y_L=EY$ and $Y_R=Y$, then $Z_L=-EZ$ and $Z_R=Z$ satisfy the
2sGRQI equations~\eqref{eq:GRQI-LR} with
\begin{equation} \label{eq:GRQI-Eskew}
\boxed{CZ- Z\,(Y^HEY)^{-1} (Y^HECY) = Y.}
\end{equation}
Therefore, if $\calY_L=E\calY_R$, then $\calY_{L+} = E\calY_{R+}$. \\
Moreover, let $\calV$ be a \emph{full} nondefective right-eigenspace
of $C$ (which means that the eigenvalues of $C|_\calV$ have the same
multiplicity as in $C$, the spectrum of $C|_\calV$ is symmetric with
respect to the imaginary axis, and $C|_\calV$ is nondefective). Then
the subspace iteration $\eqcllo{Y}\mapsto \eqcllo{Z}$ defined
by~\eqref{eq:GRQI-Esym} converges locally cubically to $\calV$.
\end{prpstn}
Note that this proposition differs from
Proposition~\ref{thm:GRQI-Esym} in two points: $Z_L = -EZ$ and the
specification that $\calV$ must be full.

\begin{proof}
It is easy to check that replacing $Y_R:=Y$, $Z_R:=Z$, $Y_L:=EY_R$,
$Z_L:=-EZ_R$ in~\eqref{eq:GRQI-R} and~\eqref{eq:GRQI-L}
yields~\eqref{eq:GRQI-Eskew} in both cases. In order to prove cubic
convergence, it is sufficient to notice that the pairs
$(\calV_L,\calV_R)$ of full nondefective left-right eigenspaces satisfy
$\calV_L=E\calV_R$, as was shown above. Therefore, if $\calY$ is close
to $\calV_R$, then the pair $(\calY_L,\calY_R):=(E\calY,\calY)$ is close
to $(\calV_L, \calV_R)$ and local cubic convergence to $\calV$ follows
from Theorem~\ref{thm:cubic-2sGRQI-S}.
\end{proof}

\subsubsection{Skew-Hermitian eigenproblem}

Let $\Omega$ be skew-Hermitian. Then we have $EC=-C^HE$ with $C:=\Omega$
and $E:=I$. The modified GRQI equation~\eqref{eq:GRQI-Eskew} becomes
\begin{equation} \label{eq:GRQI-skew}
\boxed{
\Omega Z-Z\,(Y^HY)^{-1}(Y^H\Omega Y)=Y.}
\end{equation}
This is simply the classical GRQI equation~\eqref{eq:GRQI}. This is not
surprising as skew-Hermitian matrices are normal matrices.

\subsubsection{Hamiltonian eigenproblem}

Let $H$ be Hamiltonian, i.e., $(HJ)^H = HJ$, where
$J=\left(\begin{smallmatrix} 0 & I \\ -I & 0
\end{smallmatrix}\right)$. This is equivalent to $J H = -H^H
J$. Thus we have $EC=-C^HE$ with $C:=H$ and $E:=J$, and the modified GRQI
equation~\eqref{eq:GRQI-Eskew} reads
\begin{equation} \label{eq:GRQI-Ham}
\boxed{HZ-Z\,(Y^HJY)^{-1}(Y^HJHY)=Y.}
\end{equation}
Proposition~\ref{thm:GRQI-Eskew} implies that the subspace iteration with
iteration mapping $\eqcllo{Y}\mapsto\eqcllo{Z}$ defined by~\eqref{eq:GRQI-Ham}
converges locally cubically to the full nondefective right-eigenspaces of $H$.

\subsection{The generalized eigenvalue problem}

We briefly discuss the application of the 2sGRQI concept to the
generalized eigenvalue problem. Let $A,B\in\cc^{n\times n}$. The
generalized eigenvalue problem consists in finding the nontrivial
solutions of the equation $Ax=\lambda Bx$. Corresponding to the notion
of invariant subspace for a single matrix, we have the notion of a
\emph{deflating subspace}, see e.g.~\cite{Ste73,GV96}. The
$p$-dimensional subspace $\calX$ is deflating for the pencil
$A-\lambda B$ if there exists a $p$-dimensional subspace $\calY$ such
that
\begin{equation}  \label{eq:deflating-subspace}
A\calX, B\calX \ \subset \ \calY.
\end{equation}
Here we suppose that the pencil $A-\lambda B$ is nondegenerate,
i.e., $\det(A-\lambda B)$ is not trivially zero. Then there exists
$\alpha$ and $\beta$ such that $\hat B := \alpha B - \beta A$ is
invertible. Now take $\gamma$, $\delta$ such that
$\alpha\delta-\gamma\beta \neq 0$ and let $\hat A := \gamma B -
\delta A$.  Then~\eqref{eq:deflating-subspace} is equivalent to
\begin{gather*}
\hat B^{-1} \hat A \calX \subset \calX \\
\hat B \calX = \calY,
\end{gather*}
i.e., $\calX$ is an invariant subspace of $\hat B^{-1} \hat A$.
Replacing this expression for $C$ in~\eqref{eq:GRQI-LR}, one
obtains after some manipulations
\begin{subequations}  \label{eq:2sGRQI-pencil}
\begin{gather}
\hat A Z_R \hat Y_L^H \hat B Y_R - \hat B Z_R \hat Y_L^H \hat A
Y_R = \hat B Y_R
\\ \hat A^H \hat Z_L Y_R^H \hat B^H \hat Y_L -
\hat B^H \hat Z_L Y_R^H \hat A^H \hat Y_L = \hat B^H \hat Y_L
\end{gather}
\end{subequations}
where $\hat Y_L := \hat B^{-H} Y_L$ and $\hat Z_L := \hat B^{-H}
Z_L$. It yields an iteration for which $Y_R$ and $\hat Y_L$
locally cubically converge to pairs of left-right deflating
subspaces of the pencil $A-\lambda B$. Note that if $B$ is
invertible then we can choose $\hat B:=B$ and $\hat A:=A$.

\section{Numerical experiments}
\label{sec:numerical-experiments}


We report on numerical experiments that illustrate the potential of
the 2sGRQI method (Algorithm~\ref{al:2sGRQI}) as a numerical
algorithm. The 2sGRQI method has been implemented in Matlab as
described below.

\begin{lgrthm}[implementation of 2sGRQI]
\label{al:num-2sGRQI}
Let $C$ be an $n\times n$ matrix. Given two $n\times p$ matrices
$Y_{L_0}$ and $Y_{R_0}$ satisfying
$Y_{L_0}^HY_{L_0}=I=Y_{R_0}^HY_{R_0}$, the algorithm produces a
sequence of matrices $(Y_{L_k},Y_{R_k})$ as follows. For
$k=0,1,2,\ldots$, \\ 1. Compute the $p\times p$ block Rayleigh
quotient $R_R:=(Y_{L_k}^HY_{R_k})^{-1} Y_{L_k}^HCY_{R_k}$. Compute an
eigendecomposition $R_R = W_R \diag(\rho_1,\ldots,\rho_p) W_R^{-1}$
using the Matlab {\tt eig} function. Obtain the eigendecomposition
$R_L = W_L^H \diag(\rho_1,\ldots,\rho_p) W_L^{-H}$ by computing
$W_L:=(Y_{L_k}^HY_{R_k}) W_R$.  \\ 2. Solve the decoupled
equations~\eqref{eq:GRQI-LR}, that is, $(C-\rho_i I) (z_R)_i =
Y_{R_k}W_R e_i$ and $(C^H-\rho_iI)(z_L)_i = Y_{L_k}W_L^{-H}e_i$,
$i=1,\ldots,p$, using the Matlab ``{\tt $\backslash$}'' operator. If
the solutions have any nonfinite element, then solve instead
$(C-\rho_i I+\epsilon I) (z_R)_i = Y_{R_k}W_R e_i$ and
$(C^H-\rho_iI+\epsilon I)(z_L)_i = Y_{L_k}W_L^{-H}e_i$,
$i=1,\ldots,p$, with $\epsilon$ small (we took $\epsilon=
10^3\mathbf{u}\|C\|_F$ where $\mathbf{u}$ is the floating point
relative accuracy and $\|C\|_F$ is the Frobenius norm of $C$).
\pacomm{(Check)} \\ 3.  Orthonormalize $Z_R:=\begin{bmatrix}(z_R)_1 &
  \cdots & (z_R)_i\end{bmatrix}$ to obtain $Y_{R_{k+1}}$, and likewise
for $Z_R$ to obtain $Y_{L_{k+1}}$.  In Matlab, orthonormalizations are
performed using the ``economy size'' QR decomposition, {\tt
  [YL,ignore] = qr(ZL,0)} and {\tt [YR,ignore] = qr(ZR,0)}.
\end{lgrthm}
Note that if $C$, $Y_{L_0}$ and $Y_{R_0}$ are real, then the columns
of $Z_L$ and $Z_R$ appear in complex conjugate pairs and unnecessary
work can thus be avoided in the computation of $Z_L$ and $Z_R$.

\pacomm{(This is an old version, with Schur form and back
substitution. We notice that the large numerical errors in the norm of
the columns of $Z$ first computed pollute the computation of the
subsequent columns of $Z$. It is better to decouple the equations by
diagonalizing the small block Rayleigh quotient $R_R$.)
\begin{lgrthm}[implementation of 2sGRQI]
\label{al:num-2sGRQI-old}
Let $C$ be an $n\times n$ matrix. Given two $n\times p$ matrices
$Y_{L_0}$ and $Y_{R_0}$ satisfying
$Y_{L_0}^HY_{L_0}=I=Y_{R_0}^HY_{R_0}$, the algorithm produces a sequence
of matrices $(Y_{L_k},Y_{R_k})$ as follows. For $k=1,2,\ldots$, \\
1. Apply the real QZ algorithm on the pair
$(Y_{R_k}^HC^HY_{L_k},Y_{R_k}^HY_{L_k})$ to obtain the pair $(T,R)$
where $T$ is quasi-upper-triangular and $R$ is upper-triangular. In
Matlab, {\tt [T,R,Q,U] \linebreak[4]= qz(YL'*C'*YL,YR'*YL,'real')}. \\ 2. Solve
the equations $CZ_R R^H - Z_R T^H = Y_RQ^H$ and $C^H Z_L R - Z_L T =
Y_LU$ by successive substitution as explained in
Section~\ref{sec:practical-implementation}. The $n\times n$ and $2n\times
2n$ linear systems are solved in Matlab using the operator {\tt
$\backslash$}. \\ 3. Orthonormalize $Z_L$ and $Z_R$ to obtain
$Y_{L_{k+1}}$ and $Y_{R_{k+1}}$. In Matlab, {\tt [YL,dummy] = qr(ZL,0)}
and {\tt [YR,dummy] = qr(ZR,0)}.
\end{lgrthm}
}
 
It is well known~\cite{BS89} that the basins of attraction of RQI
(Algorithm~\ref{al:RQI}) may collapse around attractors when the
eigenvalues of $A$ are not well separated. This property also holds
for GRQI~\cite{ASVM2004-01} and obviously extends to 2sGRQI
(Algorithm~\ref{al:2sGRQI}). Moreover, in 2sGRQI the matrix $C$ is not
necessarily Hermitian; its eigenspaces can thus be arbitrarily close
to each other. In a first set of experiments, in order to ensure a
reasonably large basin of attraction around the left-right
eigenspaces, we ruled out clustered eigenvalues and ill-separated
eigenvectors by choosing $C$ as follows: $C=SDS^{-1}$, where $D$ is a
diagonal matrix whose diagonal elements are random permutations of
$1,\ldots,n$ and $S = I+\tfrac{\alpha}{\|E\|_2}E$, where the elements
of $E$ are observations of independent random variables with standard
normal distribution and $\alpha$ is chosen from the uniform
distribution on the interval $(0,0.1)$. The initial matrices $Y_{L_0}$
and $Y_{R_0}$ are randomly chosen such that
$\dist(\ecl{Y_{R_0}},\ecl{S(:,1:p)})<0.1$ and
$\dist(\ecl{Y_{L_0}},\ecl{S^{-H}(:,1:p)})<0.1$, where ``$\dist$'' is
the largest principal angle.

Algorithm~\ref{al:num-2sGRQI} was run $10^6$ times with $n=20$, $p=5$.
The matrices $C$, $Y_{L_0}$, and $Y_{R_0}$ were randomly chosen in
each experiment as explained above. Experiments were run using Matlab
7.2 with floating point relative accuracy approximately equal to
$2\cdot 10^{-16}$.  Results are summarized in Table~\ref{tab:conv},
where the error $e$ is defined as the largest principal angle between
$\ecl{Y_R}$ and $\ecl{S(:,1:p)}$ plus the largest principal angle
between $\ecl{Y_L}$ and $\ecl{S^{-H}(:,1:p)}$. These results show that
convergence to the target eigenspace occurred in each of the $10^6$
runs. The evolution of the error is compatible with cubic order of
convergence.

\begin{table}[ht]
\begin{center}
\begin{tabular}{|c|c|c|}
\hline
Iterate number & mean(log10(e)) & max(log10(e)) \\ \hline
0 &  -1.4338 &  -1.0000\\
1 &  -4.6531 &  -2.6338\\
2 & -13.9359 &  -8.3053\\
3 & -16.5507 & -15.1861\\
4 & -16.5524 & -15.1651\\
5 & -16.5509 & -15.1691\\ \hline
\end{tabular}
\end{center}
\caption{Numerical experiments for Algorithm~\ref{al:num-2sGRQI}. See
details in the text.}
\label{tab:conv}
\end{table}

\pacomm{(The table below corresponds to experiments where I was
solving the 2sGRQI equations using a Schur form and back
substitution, if I am not mistaken. This would explain the rather large
worse case error (see right-hand column). The error would be smaller
with diagonalization of the small matrix, instead of triangularization.)
\begin{table}[ht]
\begin{center}
\begin{tabular}{|c|c|c|}
\hline
Iterate number & mean(log10(e)) & max(log10(e)) \\ \hline
0 & -1.0482 & -0.6996 \\
1 & -4.2584 & -2.4413 \\
2 & -13.5275 & -7.9426 \\
3 & -16.2589 & -11.0974 \\
4 & -16.2572 & -11.5222 \\
5 & -16.2567 & -11.3846  \\ \hline
\end{tabular}
\end{center}
\caption{Numerical experiments for Algorithm~\ref{al:num-2sGRQI}. See
details in the text.}
\label{tab:conv-old}
\end{table}
}

The behavior of the 2sGRQI algorithm in case of ill-separated
eigenvectors/values would deserve investigation. The Hermitian case is
studied in~\cite{ASVM2004-01} where improvements of GRQI and the
Riemannian Newton algorithm are proposed.

In another set of experiments, real Hamiltonian matrices $C$ were selected
randomly as 
\[
C = \begin{bmatrix} F & \tilde G + \tilde G^H \\ \tilde H + \tilde H^H
  & -F^H \end{bmatrix} 
\]
where $F$, $\tilde G$ and $\tilde H$ are matrices of dimension
$\tfrac{n}{2}\times\tfrac{n}{2}$ whose elements are independent
observations of the standard normally distributed random variable. A
new matrix $C$ was selected for each experiment. For testing purposes,
an eigenvalue decomposition $C=SDS^{-1}$ was computed using the Matlab
{\tt eig} function, and the \emph{full} left and right real
eigenspaces corresponding to the eigenvalues with largest real part in
magnitude were chosen as the target left and right eigenspaces. (The
notion of \emph{full} eigenspace is defined in
Section~\ref{sec:EC=-CTE}. The \emph{real} eigenspace associated to a
pair $(\lambda,\overline{\lambda})$ of complex conjugate eigenvalues
with eigenvectors $v_r+iv_i$ and $v_r-iv_i$ is the span of $v_r$ and
$v_i$.) The eigenvalue decomposition was ordered in such a way that
$\ecl{S^{-H}(:,1:p})$ is the target left-eigspace and $\ecl{S(:,1:p)}$
is the target right-eigenspace. Note that we have $p=2$ when the
target eigenvalues are real ($\lambda$ and $-\lambda$), or $p=4$ when
the target eigenvalues have a nonzero imaginary part ($\lambda$,
$\overline{\lambda}$, $-\lambda$, and $-\overline{\lambda}$). The
initial matrix $Y_{R_0}$ was randomly chosen such that
$\dist(\ecl{Y_{R_0}},\ecl{S(:,1:p)})<0.1$, where ``$\dist$'' is the
largest principal angle, and $Y_{L_0}$ was chosen as $JY_{R_0}$ in
accordance with the material of Section~\ref{sec:EC=-CTE}.
Convergence to the target left and right eigenspaces was declared when
the error $e$ as defined above was smaller than $10^{-12}$ at the 10th
iterate.  Algorithm~\ref{al:num-2sGRQI} was run $10^6$ times with
$n=20$ and $p=4$ with the matrices $C$, $Y_{L_0}$ and $Y_{R_0}$
randomly chosen in each experiment as described above. Note that, in
accordance with the material in Section~\ref{sec:EC=-CTE}, only $Z_R$
was computed at each iteration; $Z_L$ was chosen as $JZ_R$. We
observed that convergence to the target eigenspaces was declared for
$99.95\%$ of the $10^6$ experiments.
Next, the experiment was run $10^6$ times with the distance bound on
the initial condition set to $0.001$ instead of $0.1$. Convergence to
the target eigenspaces was declared for all but seven of the $10^6$
randomly generated experiments. This confirms the potential of
Algorithm~\ref{al:num-2sGRQI} for refining initial estimates of
\emph{full eigenspaces} of Hamiltonian matrices.

\section{Conclusion}
\label{sec:concl}

We have shown that Ostrowski's two-sided iteration generalizes to
an iteration on $\Grasspn\times\Grasspn$ that converges locally
cubically to the pairs of spectral nondefective left-right
eigenspaces of arbitrary square matrices. The iteration is competitive with
Chatelin's Newton method and it yields one-sided formulations adapted
to some structured eigenproblems, including the Hamiltonian and
generalized Hermitian eigenproblems.

\pacomm{(I remove this.) 
In this paper, all matrices used in the iterations were real even
though we used complex matrices in some of the proofs. When going
to complex arithmetic everything becomes much simpler. The QZ
algorithm gives triangular matrices, the solution of the Sylvester
equations are simpler recurrences but implying complex vectors and
systems of equations.
}

\section*{Acknowledgements}
This work was initiated when the first author was a guest in
the Mathematics Department of the University of W\"urzburg under a grant
from the European Nonlinear Control Network. The hospitality of the
members of the Department is gratefully acknowledged. The first author
would also like to thank especially Rodolphe Sepulchre for his careful
guidance throughout his PhD research and the impact this had on the
results of this paper.

\bibliographystyle{amsalpha_compact}
\bibliography{pabib}

\end{document}